\documentclass[aps,pra,groupedaddress,amsmath,amssymb,tightenlines,notitlepage,nofootinbib,onecolumn,superscriptaddress,11pt]{revtex4-1}
\pdfoutput=1

\PassOptionsToPackage{pdftitle={Postselected quantum hypothesis testing}}{hyperref}
\usepackage{bm,bbm}% 
\usepackage{braket}
\usepackage{amsthm,amsmath,amssymb}

\usepackage[utf8]{inputenc}

\usepackage[dvipsnames,svgnames, table]{xcolor}
\usepackage{amsmath,amssymb}
\usepackage{changepage,afterpage}
\usepackage{etruscan}
\usepackage{tikz}
\usepackage{enumerate}
\usepackage{mathtools}
\usepackage[colorlinks=true, allcolors=MidnightBlue!70!black!70!TealBlue]{hyperref}

\usepackage{array,booktabs,multirow}
\usepackage{siunitx}

\usepackage{newpxtext,newpxmath}

\newtheorem{theorem}{Theorem}

\newtheorem{proposition}[theorem]{Proposition}

\newtheorem{example}[theorem]{Example}

\theoremstyle{definition}
\newtheorem*{remark}{Remark}

\newtheoremstyle{red}{}{}{\normalfont}{}{\color{red!80!black}\bfseries}{.}{ }{}
\theoremstyle{red}

\let\nc\newcommand

\nc{\NegW}{W_\tau}

\nc{\RW}{\Omega_\FF}

\nc{\id}{\mathbbm{1}}

\renewcommand{\O}{\mathcal{O}}

\newcommand{\DD}{\mathbb{D}}

\newcommand{\FF}{\mathcal{F}}

\nc{\Dmax}{D_{\max}}

\let\O\OO

\nc{\SEP}{\mathrm{SEP}}
\nc{\STAB}{\mathrm{STAB}}
\nc{\PPT}{\mathrm{PPT}}
\nc{\PPTP}{\mathrm{PPTP}}
\nc{\SEPP}{\mathrm{SEPP}}
\nc{\SP}{\mathrm{KP}}
\nc{\FP}{\mathrm{FP}}
\nc{\CPTP}{{\mathrm{CPTP}}}

\nc{\Op}{\mathcal{O}}

\nc{\idc}{\mathrm{id}}
\nc{\ve}{\varepsilon}

\nc{\Omax}{\O_{\mathrm{max}}}

\usepackage{stackengine,moresize}
\stackMath

\nc{\sminfty}{{\infty,\bullet}}

\usepackage[most,breakable]{tcolorbox}
  {\expandafter\ifstrequal\expandafter{#1}{orange}{\begin{tcolorbox}[colback=orange!5,colframe=orange!15,breakable,enhanced]}{\begin{tcolorbox}[colback=white,colframe=gray!10,breakable,enhanced]}}%
  {\end{tcolorbox}}

\nc{\regrob}{\DD_{s,\FF}^{\infty}}
\nc{\regrobs}{\DD_{s,\FF}^{\sminfty}}

\makeatletter 
\renewcommand\onecolumngrid{% 
\do@columngrid{one}{\@ne}%
\def\set@footnotewidth{\onecolumngrid}% <<<<<<<<<<<<<<<<
\def\footnoterule{\kern-6pt\hrule width 1.5in\kern6pt}%
}
\makeatother

\usepackage{fmtcount,refcount}

\let\oldproofname\proofname
\renewcommand{\proofname}{\rm\bf{\oldproofname}}

\usepackage{pifont}

\usepackage{newunicodechar}

\def\id{\operatorname{id}}

\def\1{\openone}

\def\keywords{\xdef\@thefnmark{}\@footnotetext}

\def\DD{D}

\allowdisplaybreaks

\begin{document}

\title{Block perturbation of symplectic matrices in Williamson's theorem}

    \author{Gajendra Babu}
    \email{gajendra0777@gmail.com}
     \affiliation{Department of Mathematics, GLA University, Mathura 281406, India}

    \author{Hemant K. Mishra}
    \email{hemant.mishra@cornell.edu}
    \affiliation{Theoretical Statistics and Mathematics Unit, Indian Statistical Institute, Delhi 110016, India}
    \affiliation{School of Electrical and Computer Engineering, Cornell University, Ithaca, New York 14850, USA}

\begin{abstract}
           Williamson's theorem states that for any $2n \times 2n$ real positive definite matrix $A$, there exists a $2n \times 2n$ real symplectic matrix $S$ such that $S^TAS=D \oplus D$,  where $D$ is an $n\times n$ diagonal matrix with positive diagonal entries
            which are known as the symplectic eigenvalues of $A$. Let $H$ be any $2n \times 2n$ real symmetric matrix such that the perturbed matrix $A+H$ is also positive definite.  In this paper, we show that any symplectic matrix $\tilde{S}$ diagonalizing $A+H$ in Williamson's theorem is of the form $\tilde{S}=S Q+\mathcal{O}(\|H\|)$, where $Q$ is a $2n \times 2n$ real symplectic as well as orthogonal matrix. Moreover, $Q$ is in \emph{symplectic block diagonal} form with the block sizes given by twice the multiplicities of the symplectic eigenvalues of $A$. Consequently, we show that $\tilde{S}$ and $S$ can be chosen so that $\|\tilde{S}-S\|=\mathcal{O}(\|H\|)$. Our results hold even if $A$ has repeated symplectic eigenvalues. This generalizes the stability result of symplectic matrices for non-repeated symplectic eigenvalues given by Idel, Gaona, and Wolf [\emph{Linear Algebra Appl., 525:45-58, 2017}].
\end{abstract}

\maketitle

%%%%%%%%%%%%%%%%%%%%%%%%%%%%%%%%%%%%%%%%%%%%%%%%%%%%%%%%%%%%%%%%%%
%%%%%%%%%%%%%%%%%%%%%%%%%%%%%%%%%%%%%%%%%%%%%%%%%%%%%%%%%%%%%%%%%%
%\tableofcontents
%%%%%%%%%%%%%%%%%%%%%%%%%%%%%%%%%%%%%%%%%%%%%%%%%%%%%%%%%%%%%%%%%%

\vskip0.3in
%\subjclass[2000] {15A48, 47B34.}

 
\section{Introduction}
%The spectral theorem for Hermitian matrices is a powerful tool with implications in many areas of mathematics such as combinatorial graph theory \cite{chung1997spectral, beineke2004topics, spielman2012spectral}, numerical analysis \cite{saad2011numerical, bekas2005computation,clavero2011uniformly}, optimization theory \cite{ebadian2011perspectives}, engineering \cite{pillai2005perron, bock2023jordan, truong2020theory}, and most areas of theoretical physics \cite{prosen2010spectral}.
%The wide applicability of the spectral theorem can be mostly attributed to the deep understanding of perturbation theory of eigenvalues and eigenspaces of Hermitian matrices \cite{rellich1969perturbation, kato2013perturbation}, and the results are well-known today \cite{ma_bhatia}.

Analogous to the spectral theorem in linear algebra is Williamson's theorem \cite{williamson1936algebraic} in symplectic linear algebra. It states that for any $2n \times 2n$ real positive definite matrix $A$, there exists a $2n \times 2n$ real {\it symplectic matrix} $S$ such that $S^TAS=D \oplus D$, where $D$ is an $n\times n$ diagonal matrix with positive diagonal entries. The diagonal entries of $D$ are known as the {\it symplectic eigenvalues} of $A$, and the columns of $S$ form a {\it symplectic eigenbasis} of $A$. This result is also referred to as {\it Williamson normal form} in the literature \cite{degosson, dms}.
Symplectic eigenvalues and symplectic matrices are ubiquitous in many areas such as classical Hamiltonian dynamics \cite{arnold1989graduate}, quantum mechanics \cite{dms}, and symplectic topology \cite{hofer}. More recently, it has attracted much attention from matrix analysts \cite{bhatia2015symplectic, bhatia2020schur, mishra2020first, bhatia2021variational, jain2021sums, jm, son2022symplectic, paradan2022horn} and quantum physicists \cite{adesso2004extremal, chen2005gaussian, idel, nicacio2021williamson, hsiang2022entanglement} for its important role in continuous-variable quantum information theory \cite{serafini2017quantum}. For example, any Gaussian state of zero mean vector is obtained by applying to a tensor product of thermal states a unitary map that is characterized by a symplectic matrix \cite{serafini2017quantum}, and the von-Neumann entropy of the Gaussian state is a smooth function of the symplectic eigenvalues of its covariance matrix \cite{p}. 
So, it is of theoretical interest as well as practical importance to study the perturbation of symplectic eigenvalues and symplectic matrices in Williamson's theorem, both of which are closely related to each other. Indeed, the perturbation bound on symplectic eigenvalues of two positive definite matrices $A$ and $B$ obtained in \cite{jm} is derived using symplectic matrices diagonalizing $tA+(1-t)B$ for $t\in [0,1]$. In \cite{idel}, a perturbation of $A$ of the form $A+tH$ was considered for small variable $t > 0$ and a fixed real symmetric matrix $H$. The authors studied the stability of symplectic matrices diagonalizing $A+tH$ in Williamson's theorem and a perturbation bound was obtained for the case of $A$ having non-repeated symplectic eigenvalues. 

In this paper, we study the stability of symplectic matrices in Williamson's theorem diagonalizing $A+H$, where $H$ is an arbitrary $2n \times 2n$ real symmetric matrix such that the perturbed matrix $A+H$ is also positive definite. 
Let $S$ be a fixed symplectic matrix diagonalizing $A$ in Williamson's theorem. 
We show that any symplectic matrix $\tilde{S}$ diagonalizing $A+H$ in Williamson's theorem is of the form $\tilde{S}=S Q+\mathcal{O}(\|H\|)$ such that $Q$ is a $2n \times 2n$ real symplectic as well as orthogonal matrix. Moreover, $Q$ is in \emph{symplectic block diagonal} form with block sizes given by twice the multiplicities of the symplectic eigenvalues of $A$.
Consequently, we prove that $\tilde{S}$ and $S$ can be chosen so that $\|\tilde{S}-S\|=\mathcal{O}(\|H\|)$. Our results hold even if $A$ has repeated symplectic eigenvalues, generalizing the stability result of symplectic matrices corresponding to the case of non-repeated symplectic eigenvalues given in \cite{idel}. We do not provide any perturbation bounds.

The rest of the paper is organized as follows. In Section~\ref{back}, we review some definitions, set notations, and define basic symplectic operations. In
Section~\ref{main_sec}, we detail the findings of this paper. These are given in Proposition~\ref{res1}, Theorem~\ref{res2}, Theorem~\ref{res4}, and Proposition~\ref{res3}.

%%%%%%%%%%%%%%%%%%%%%%%%%%%%%%%%%%%%%%%%%%%%%%%%%%%%%%%%%%%%%%%%%%

\section{Background and notations}\label{back}
Let $\operatorname{Sm}(m)$ denote the set of $m \times m$ real symmetric matrices equipped with the spectral norm $\|\cdot \|$, that is, for any $X \in \operatorname{Sm}(m)$, $\|X\|$ is the maximum singular value of $X$.
We also use the same notation $\|\cdot \|$ for the Euclidean norm, and $\langle \cdot, \cdot \rangle$ for the Euclidean inner product on $\mathbb{R}^{m}$ or $\mathbb{C}^m$.
Let $0_{i,j}$ denote the $i \times j$ zero matrix, and let $0_i$ denote the $i\times i$ zero matrix (i.e., $0_i=0_{i,i}$).
We denote the imaginary unit number by $\iota \coloneqq \sqrt{-1}$. 
We use the Big-O notation $Y=\mathcal{O}(\|X\|)$ to denote a matrix $Y$ as a function of $X$ for which there exist positive scalars $c$ and $\delta$ such that $\|Y\| \leq c \|X\|$ for all $X$ with $\|X\| < \delta$.

%%%%%%%%%%%%%%%%%%%%%%%%%%%%%%%%%%%%%%%%%%%%%%%%%%%%%%%%%%%%%%%%%%

\subsection{Symplectic matrices and symplectic eigenvalues}
Define $J_2\coloneqq \left(\begin{smallmatrix} 0 & 1 \\ -1 & 0 \end{smallmatrix}\right)$, and let $J_{2n}=J_2 \otimes I_n$ for $n>1$, where  $I_n$ is the $n \times n$ identity matrix.
A $2n \times 2n$ real matrix $S$ is said to be symplectic if $S^TJ_{2n}S=J_{2n}.$
The set of $2n \times 2n$ symplectic matrices, denote by $\operatorname{Sp}(2n)$, forms a group under multiplication called the {\it symplectic group}.   
The symplectic group $\operatorname{Sp}(2n)$ is analogous to the orthogonal group $\operatorname{Or}(2n)$ of $2n \times 2n$ orthogonal matrices in the sense that replacing the matrix $J_{2n}$ with $I_{2n}$ in the definition of symplectic matrices gives the definition of orthogonal matrices.
However, in contrast with the orthogonal group, the symplectic group is non-compact. Also, the determinant of every symplectic matrix is equal to $+1$ which makes the symplectic group a subgroup of the special linear group \cite{dms}.
Let $\operatorname{Pd}(2n) \subset \operatorname{Sm}(2n)$ denote the set of positive definite matrices. Williamson's theorem \cite{williamson1936algebraic} states that for every $A \in \operatorname{Pd}(2n)$,  there exists $S\in \operatorname{Sp}(2n)$ such that $S^TAS=D \oplus D$, where $D$ is an $n\times n$  diagonal matrix. The diagonal elements $d_1(A) \le \cdots \le d_n(A)$ of $D$ are independent of the choice of $S$, and they are known as the symplectic eigenvalues of $A$. 
Denote by $\operatorname{Sp}(2n; A)$ the subset of $\operatorname{Sp}(2n)$ consisting of symplectic matrices that diagonalize $A$ in Williamson's theorem.
Several proofs of Williamson's theorem are available using basic linear algebra (e.g., \cite{degosson, simon1999congruences}).

Denote the set of $2n \times 2n$ {\it orthosymplectic} (orthogonal as well as symplectic) matrices as $\operatorname{OrSp}(2n)\coloneqq \operatorname{Or}(2n) \cap \operatorname{Sp}(2n)$.
Any orthosymplectic matrix $Q \in \operatorname{OrSp}(2n)$ is precisely of the form 
\begin{align}\label{eq:orthosymplectic_matrix_representation}
    Q = \begin{pmatrix}
            X & Y \\
            -Y & X
        \end{pmatrix},
\end{align}
where $X,Y$ are $n\times n$ real matrices such that $X+\iota Y$ is a unitary matrix \cite{bhatia2015symplectic}.
For $m \leq n$, we denote by $\operatorname{Sp}(2n, 2m)$ the set of $2n \times 2m$ matrices $M$ satisfying $M^T J_{2n} M = J_{2m}$. In particular, we have $\operatorname{Sp}(2n, 2n)=\operatorname{Sp}(2n)$.

%%%%%%%%%%%%%%%%%%%%%%%%%%%%%%%%%%%%%%%%%%%%%%%%%%%%%%%%%%%%%%%%%%
\subsection{Symplectic block and symplectic direct sum}

Let $m$ be a natural number and $\mathcal{I}, \mathcal{J} \subseteq \{1, \ldots, m\}$. Suppose $M$ is an $m\times m$ matrix. We
denote by  $M_{\mathcal{J}}$ the submatrix of $M$ consisting of the columns of $M$ with indices in $\mathcal{J}$.
Also, denote by $M_{\mathcal{I} \mathcal{J}}$ the $|\mathcal{I}| \times |\mathcal{J}|$ submatrix of $M=[M_{ij}]$ consisting of the elements $M_{ij}$ with indices $i\in \mathcal{I}$ and $j\in \mathcal{J}$. 
Let $T$ be any $2m \times 2m$ matrix given in the block form by
\begin{align*}
    T = 
        \begin{pmatrix}
            W & X \\
            Y & Z
        \end{pmatrix},   
\end{align*}
where $X,Y,W,Z$ are matrices of order $m \times m$. Define a {\it symplectic block} of $T$ as a submatrix of the form
\begin{align*}
         \begin{pmatrix}
            W_{\mathcal{I}\mathcal{J}} & X_{\mathcal{I}\mathcal{J}} \\
            Y_{\mathcal{I}\mathcal{J}} & Z_{\mathcal{I}\mathcal{J}}
        \end{pmatrix}.
\end{align*}
Also, define a {\it symplectic diagonal block} of $T$ as a submatrix of the form
\begin{align*}
         \begin{pmatrix}
            W_{\mathcal{I}\mathcal{I}} & X_{\mathcal{I}\mathcal{I}} \\
            Y_{\mathcal{I}\mathcal{I}} & Z_{\mathcal{I}\mathcal{I}}
        \end{pmatrix}.
\end{align*}
The following example illustrates this.
\begin{example}
    Let $T$ be a $6 \times 6$ matrix given by
\begin{align*}
     T=\left(
    \begin{array}{ccccccc}
     \cline{1-2}\cline{5-6} 
      \multicolumn{1}{|c}{1} & \multicolumn{1}{c|}{2} & 3 & \vdots & \multicolumn{1}{|c}{4} & \multicolumn{1}{c|}{5} & 6 \\ 
      \multicolumn{1}{|c}{7} & \multicolumn{1}{c|}{8} & 9 & \vdots & \multicolumn{1}{|c}{10} & \multicolumn{1}{c|}{11} & 12 \\ \cline{1-2} \cline{5-6} 
      13 & \cellcolor{gray!25} 14 & 15 & \vdots & 16 & \cellcolor{gray!25} 17 & 18 \\ 
      \hdots & \hdots &\hdots &\hdots &\hdots &\hdots &\hdots 
      \\ [-0.1cm] \cline{1-2} \cline{5-6}
      \multicolumn{1}{|c}{19} & \multicolumn{1}{c|}{20} & 21 & \vdots & \multicolumn{1}{|c}{22} & \multicolumn{1}{c|}{23} & 24 \\ 
      \multicolumn{1}{|c}{25} & \multicolumn{1}{c|}{26} & 27 & \vdots & \multicolumn{1}{|c}{28} & \multicolumn{1}{c|}{29} & 30 \\ \cline{1-2} \cline{5-6}
      31 & \cellcolor{gray!25} 32 & 33 & \vdots & 34 & \cellcolor{gray!25} 35 & 36
       \end{array}
    \right).
\end{align*}
A symplectic block of $T$, which corresponds to $\mathcal{I}=\{3\}$ and $\mathcal{J}=\{2\}$, is given by
\begin{align*}
    \left(
    \begin{array}{cc}
      \cellcolor{gray!25} 14 & \cellcolor{gray!25} 17 \\
      \cellcolor{gray!25} 32 & \cellcolor{gray!25} 35
    \end{array}
    \right).
\end{align*}
A symplectic diagonal block, corresponding to $\mathcal{I}=\{1,2\}$, is given by
\begin{align*}
    \left(
    \begin{array}{cccc}\cline{1-2}\cline{3-4}
      \multicolumn{1}{|c}{1} & \multicolumn{1}{c|}{2} & \multicolumn{1}{|c}{4} & \multicolumn{1}{c|}{5} \\
      \multicolumn{1}{|c}{7} & \multicolumn{1}{c|}{8} & \multicolumn{1}{|c}{10} & \multicolumn{1}{c|}{11} \\ \cline{1-2} \cline{3-4}
      \multicolumn{1}{|c}{19} & \multicolumn{1}{c|}{20} & \multicolumn{1}{|c}{22} & \multicolumn{1}{c|}{23} \\ 
      \multicolumn{1}{|c}{25} & \multicolumn{1}{c|}{26} & \multicolumn{1}{|c}{28} & \multicolumn{1}{c|}{29} \\ \cline{1-2} \cline{3-4}
       \end{array}
    \right).
\end{align*}
\end{example}
Let $T'$ be another $2m' \times 2m'$ matrix, given in the block form
\begin{align*}
    T' =
        \begin{pmatrix}
            W' & X' \\
            Y' & Z'
        \end{pmatrix},
\end{align*}
where the blocks $W',X',Y',Z'$ have size $m' \times m'$. Define the {\it symplectic direct sum} of $T$ and $T'$ as
\begin{align*}
    T \oplus^{\operatorname{s}} T' 
        &= 
           \begin{pmatrix}
            W \oplus W' & X \oplus X' \\
            Y \oplus Y' & Z \oplus Z'
        \end{pmatrix}.
\end{align*}
This is illustrated in the following example.
\begin{example}
    Let
    \begin{equation*}
        T=
            \begin{pmatrix}
            \begin{array}{cccc}\cline{1-2}\cline{3-4}
              \multicolumn{1}{|c}{1} & \multicolumn{1}{c|}{2} & \multicolumn{1}{|c}{3} & \multicolumn{1}{c|}{4} \\ 
              \multicolumn{1}{|c}{5} & \multicolumn{1}{c|}{6} & \multicolumn{1}{|c}{7} & \multicolumn{1}{c|}{8} \\ \cline{1-2} \cline{3-4}
              \multicolumn{1}{|c}{9} & \multicolumn{1}{c|}{10} & \multicolumn{1}{|c}{11} & \multicolumn{1}{c|}{12} \\ 
              \multicolumn{1}{|c}{13} & \multicolumn{1}{c|}{14} & \multicolumn{1}{|c}{15} & \multicolumn{1}{c|}{16} \\ \cline{1-2} \cline{3-4}
            \end{array}
            \end{pmatrix},
        T'=
            \left(
                    \begin{array}{cc}
                        \cellcolor{gray!25} 17 & \cellcolor{gray!25} 18 \\
                        \cellcolor{gray!25} 19 & \cellcolor{gray!25} 20
                    \end{array}
            \right).
    \end{equation*}
    We then have
    \begin{align*}
        T \oplus^{\operatorname{s}} T' =
            \begin{pmatrix}
    \begin{array}{ccccccc}\cline{1-2}\cline{5-6}
      \multicolumn{1}{|c}{1} & \multicolumn{1}{c|}{2} & 0 & \vdots & \multicolumn{1}{|c}{3} & \multicolumn{1}{c|}{4} & 0 \\
      \multicolumn{1}{|c}{5} & \multicolumn{1}{c|}{6} & 0 & \vdots & \multicolumn{1}{|c}{7} & \multicolumn{1}{c|}{8} & 0 \\ \cline{1-2} \cline{5-6}
      0 &0& \cellcolor{gray!25} 17 & \vdots & 0 & 0 &  \cellcolor{gray!25} 18 \\
      \hdots & \hdots &\hdots &\hdots &\hdots &\hdots &\hdots 
      \\ \cline{1-2} \cline{5-6}
      \multicolumn{1}{|c}{9} & \multicolumn{1}{c|}{10} & 0 & \vdots & \multicolumn{1}{|c}{11} & \multicolumn{1}{c|}{12} & 0 \\ 
      \multicolumn{1}{|c}{13} & \multicolumn{1}{c|}{14} & 0 & \vdots & \multicolumn{1}{|c}{15} & \multicolumn{1}{c|}{16} & 0 \\ \cline{1-2} \cline{5-6}
      0 & 0 & \cellcolor{gray!25} 19 & \vdots & 0 & 0 & \cellcolor{gray!25} 20
       \end{array}
    \end{pmatrix}.
    \end{align*}
\end{example}
We know that the usual direct sum of two orthogonal matrices is also an orthogonal matrix. It is interesting to note that an analogous property is also satisfied by the symplectic direct sum. If $T \in \operatorname{Sp}(2k)$ and $T' \in \operatorname{Sp}(2\ell)$, then $T \oplus^s T' \in \operatorname{Sp}(2(k+\ell))$. Indeed, we have
\begin{align*}
     &(T \oplus^s T')^T J_{2(k+\ell)} (T \oplus^s T') \\
     &\hspace{0.5cm}= \begin{pmatrix}
            W \oplus W' & X \oplus X' \\
            Y \oplus Y' & Z \oplus Z'
        \end{pmatrix}^T 
        \begin{pmatrix}
            0_{k+\ell} & I_{k+\ell} \\
            -I_{k+\ell} & 0_{k+\ell}
        \end{pmatrix}
        \begin{pmatrix}
            W \oplus W' & X \oplus X' \\
            Y \oplus Y' & Z \oplus Z'
        \end{pmatrix} \\
     &\hspace{0.5cm}= \begin{pmatrix}
            W^T \oplus W'^T & Y^T \oplus Y'^T \\
            X^T \oplus X'^T & Z^T \oplus Z'^T
        \end{pmatrix}
        \begin{pmatrix}
            Y \oplus Y' & Z \oplus Z' \\
            -(W \oplus W') & -(X \oplus X')
        \end{pmatrix} \\
    &\hspace{0.5cm}= \begin{pmatrix}
            W^TY \oplus W'^TY' - Y^TW \oplus Y'^TW' & W^TZ \oplus W'^TZ' - Y^TX \oplus Y'^TX' \\
            X^TY \oplus X'^TY'-  Z^TW \oplus Z'^TW' & X^TZ \oplus X'^TZ' - Z^TX \oplus Z'^TX'
        \end{pmatrix} \\
    &\hspace{0.5cm}= \begin{pmatrix}
            (W^TY- Y^TW) \oplus (W'^TY' - Y'^TW') & (W^TZ- Y^TX) \oplus (W'^TZ'- Y'^TX') \\
            (X^TY-  Z^TW) \oplus (X'^TY' - Z'^TW') & (X^TZ- Z^TX) \oplus (X'^TZ' - Z'^TX')
        \end{pmatrix} \\
    &\hspace{0.5cm}= \begin{pmatrix}
            W^TY- Y^TW  & W^TZ- Y^TX  \\
            X^TY-  Z^TW  & X^TZ- Z^TX
        \end{pmatrix} \oplus^s \begin{pmatrix}
             W'^TY' - Y'^TW' &  W'^TZ'- Y'^TX' \\
             X'^TY' - Z'^TW' & X'^TZ' - Z'^TX'
        \end{pmatrix} \\
    &\hspace{0.5cm}= \begin{pmatrix}
            W^T  & Y^T  \\
            X^T  & Z^T
        \end{pmatrix}
        \begin{pmatrix}
            Y  & Z  \\
            -W & -X
        \end{pmatrix}  \oplus^s 
        \begin{pmatrix}
            W'^T  & Y'^T  \\
            X'^T  & Z'^T
        \end{pmatrix}
        \begin{pmatrix}
            Y'  & Z'  \\
            -W' & -X'
        \end{pmatrix} \\
    &\hspace{0.5cm}= \begin{pmatrix}
            W  & X  \\
            Y  & Z
        \end{pmatrix}^T
        \begin{pmatrix}
            0_k  & I_k  \\
            -I_k  & 0_k
        \end{pmatrix}
        \begin{pmatrix}
            W  & X  \\
            Y & Z
        \end{pmatrix}  \oplus^s 
        \begin{pmatrix}
            W'  & X'  \\
            Y'  & Z'
        \end{pmatrix}^T
        \begin{pmatrix}
            0_{\ell}  & I_\ell  \\
            -I_\ell  & 0_{\ell}
        \end{pmatrix}
        \begin{pmatrix}
            W'  & X'  \\
            Y' & Z'
        \end{pmatrix} \\
    &\hspace{0.5cm}= T^T J_{2k}T \oplus^s T'^T J_{2\ell} T' \\
    &\hspace{0.5cm}= J_{2k} \oplus^s J_{2\ell} \\
    &\hspace{0.5cm}= J_{2(k+\ell)}.
\end{align*}
%%%%%%%%%%%%%%%%%%%%%%%%%%%%%%%%%%%%%%%%%%%%%%%%%%%%%%%%%%%%%%%%%%

\subsection{Symplectic concatenation}\label{sym_concatenation}
Let $M=\left(p_1,\ldots, p_{k}, q_1,\ldots,q_k \right)$ and $N=\left(x_1,\ldots, x_{\ell}, y_1,\ldots,y_\ell \right)$ be $2n \times 2k$ and $2n \times 2\ell$ matrices, respectively. Define the {\it symplectic concatenation} of $M$ and $N$ to be the following $2n \times 2(k+\ell)$ matrix
\begin{align*}
    M \diamond N \coloneqq \left(p_1,\ldots, p_{k},x_1,\ldots, x_{\ell},q_1,\ldots,q_k, y_1,\ldots,y_\ell  \right).
\end{align*}
Here is an example to illustrate symplectic concatenation.
\begin{example}
    Let  
    \begin{align*}
        M= \left(
            \begin{array}{cccc}
              \multicolumn{1}{c}{1} & \multicolumn{1}{c|}{2} & \multicolumn{1}{c}{3} & \multicolumn{1}{c}{4} \\
              \multicolumn{1}{c}{5} & \multicolumn{1}{c|}{6} & \multicolumn{1}{c}{7} & \multicolumn{1}{c}{8} \\ 
              \multicolumn{1}{c}{9} & \multicolumn{1}{c|}{10} & \multicolumn{1}{c}{11} & \multicolumn{1}{c}{12} \\ 
              \multicolumn{1}{c}{13} & \multicolumn{1}{c|}{14} & \multicolumn{1}{c}{15} & \multicolumn{1}{c}{16} \\ 
            \end{array}
            \right),
        N= \left(
            \begin{array}{cc}
               \cellcolor{gray!25} 17 & \cellcolor{gray!25} 18 \\
               \cellcolor{gray!25} 19 & \cellcolor{gray!25} 20 \\
               \cellcolor{gray!25} 21 & \cellcolor{gray!25} 22 \\
               \cellcolor{gray!25} 23 & \cellcolor{gray!25} 24 \\
            \end{array}
            \right).    
    \end{align*}
    The symplectic concatenation of $M$ and $N$ is given by 
    \begin{align*}
        M \diamond N = \left(
            \begin{array}{cccccc}
              \multicolumn{1}{c}{1} & \multicolumn{1}{c}{2} &  \cellcolor{gray!25} 17 &
              \multicolumn{1}{|c}{3} & \multicolumn{1}{c}{4} &
              \cellcolor{gray!25} 18 \\
              \multicolumn{1}{c}{5} & \multicolumn{1}{c}{6} &
              \cellcolor{gray!25} 19&
              \multicolumn{1}{|c}{7} & \multicolumn{1}{c}{8} &
              \cellcolor{gray!25} 20 \\ 
              \multicolumn{1}{c}{9} & \multicolumn{1}{c}{10} & 
              \cellcolor{gray!25} 21 &
              \multicolumn{1}{|c}{11} & \multicolumn{1}{c}{12} &
              \cellcolor{gray!25} 22 \\ 
              \multicolumn{1}{c}{13} & \multicolumn{1}{c}{14} &
              \cellcolor{gray!25} 23 &
              \multicolumn{1}{|c}{15} & \multicolumn{1}{c}{16} &
              \cellcolor{gray!25} 24 
            \end{array}
            \right).
    \end{align*}
\end{example}
Suppose that $M \in \operatorname{Sp}(2n, 2k)$ and $N \in \operatorname{Sp}(2n, 2\ell)$. Let us derive a necessary and sufficient condition on $M$ and $N$ for $k+\ell \leq n$ such that $M \diamond N \in \operatorname{Sp}(2n, 2(k+\ell))$. This will be useful later.
We have
\begin{align}
    (M \diamond N)^T J_{2n} (M \diamond N) &= \left((M \diamond N)^T J_{2n} M \right) \diamond \left((M \diamond N)^T J_{2n}N\right) \nonumber \\
    &=\left(M^T J_{2n}^T (M \diamond N) \right)^T \diamond \left(N^TJ_{2n}^T(M \diamond N)\right)^T  \nonumber \\
    &=\left((M^T J_{2n}^T M) \diamond (M^T J_{2n}^T N) \right)^T \diamond \left((N^TJ_{2n}^TM) \diamond (N^TJ_{2n}^TN)\right)^T \nonumber \\
    &= \left(J_{2k}^T \diamond (M^T J_{2n}^T N) \right)^T \diamond \left((N^TJ_{2n}^TM) \diamond J_{2\ell}^T\right)^T. \label{sym_concat}
\end{align}
We also observe that
\begin{align}
    J_{2(k+\ell)} &= \left(J_{2k}^T \diamond 0_{2k, 2\ell} \right)^T \diamond \left(0_{2\ell, 2k} \diamond J_{2\ell}^T\right)^T.\label{standard_sym_concat}
\end{align}
 By comparing \eqref{sym_concat} and \eqref{standard_sym_concat}, we deduce that $M \diamond N \in \operatorname{Sp}(2n, 2(k+\ell))$ if and only if $M^T J_{2n} N=0_{2k, 2\ell}$.

%%%%%%%%%%%%%%%%%%%%%%%%%%%%%%%%%%%%%%%%%%%%%%%%%%%%%%%%%%%%%%%%%%

\section{Main results}\label{main_sec}
We fix the following notations throughout the paper. Let $A \in \operatorname{Pd}(2n)$ with distinct symplectic eigenvalues $\mu_1 < \cdots < \mu_r$.
 For all $i=1, \ldots, r$, define sets
\begin{align*}
\alpha_i    
        &\coloneqq \{j : d_j(A) = \mu_i, j= 1, \ldots, n\}, \\
\beta_i     
        &\coloneqq \{j+n: j \in \alpha_i\}, \\
\gamma_i    
        &\coloneqq \alpha_i \cup \beta_i.
\end{align*}  
% The $j$th column of $H$ is denoted by $H_j$ and the $(i,j)$th entry by $H_{ij}.$
An example to illustrate these sets is as follows.
\begin{example}
    Suppose $A \in \operatorname{Pd}(20)$ with symplectic eigenvalues $1,1,2,3,3,3,4,4,4,5$. We have $\mu_1=1, \mu_2=2, \mu_3=3, \mu_4=4, \mu_5=5$. Also $\alpha_1=\{1,2\}$, $\alpha_2=\{3\}, \alpha_3=\{4,5,6\}, \alpha_4=\{7,8,9\}, \alpha_5=\{10\}$. Note that $n=10$, so we have $\beta_1=\{11,12\}$, $\beta_2=\{13\},$ $\beta_3=\{14,15,16\}$, $\beta_4=\{17,18,19\}$, $\beta_5=\{20\}$. We thus also get $\gamma_1=\{1,2,11,12\}$, $\gamma_2=\{3,13\}$, $\gamma_3=\{4,5,6,14,15,16\}$, $\gamma_4=\{7,8,9,17,18,19\}$, $\gamma_5=\{10,20\}$.
\end{example}

\begin{proposition}\label{res1}
    Let $A \in \operatorname{Pd}(2n)$ and $H\in \operatorname{Sm}(2n)$ such that $A + H \in \operatorname{Pd}(2n).$ Let $S\in \operatorname{Sp}(2n; A)$ and $\tilde{S} \in \operatorname{Sp}(2n; A+H)$. For $1\leq i\neq j\leq r$,  we have
    \begin{align}
        \left(S^{-1}\tilde{S}\right)_{\gamma_i \gamma_j}    
            &=\mathcal{O}(\|H \|), \label{nq5}\\
        \left(S^{-1}\tilde{S}\right)_{\alpha_i \alpha_i}    
            &=\left(S^{-1}\tilde{S}\right)_{\beta_i \beta_i} + \mathcal{O}(\|H \|), \label{nq6} \\
        \left(S^{-1}\tilde{S}\right)_{\alpha_i \beta_i}      
            &=-\left(S^{-1}\tilde{S}\right)_{\beta_i \alpha_i} + \mathcal{O}(\|H \|), \label{nq7} \\
        \left(S^{-1}\tilde{S}\right)_{\gamma_i \gamma_i}^T \left(S^{-1}\tilde{S}\right)_{\gamma_i \gamma_i} 
            &= I_{2|\alpha_i|} +  \mathcal{O}(\|H\|),\label{nq9}\\
        \left(S^{-1}\tilde{S}\right)_{\gamma_i \gamma_i}^T J_{2|\alpha_i|} \left(S^{-1}\tilde{S}\right)_{\gamma_i \gamma_i} 
            &= J_{2|\alpha_i|} +  \mathcal{O}(\|H\|^2). \label{nq8}
    \end{align}
\end{proposition}
\begin{proof}
It suffices to prove the assertions for $A$ in the diagonal form $A=D \oplus D$ and $S=I_{2n}$.
For any $\tilde{S} \in \operatorname{Sp}(2n; A+H)$, we have
\begin{align}\label{new2eqn1}
    \tilde{S}^T(A+H)\tilde{S} = \tilde{D} \oplus \tilde{D},
\end{align}
where $\tilde{D}$ is the diagonal matrix with entries $d_1(A+H) \leq \cdots \leq d_n(A+H)$.
By Theorem $3.1$ of \cite{idel}, we get 
\begin{align}\label{new2eqn2}
    \tilde{D}=D+\mathcal{O}(\|H\|).
\end{align} 
By \eqref{new2eqn1} and \eqref{new2eqn2}, and using the diagonal form $A=D \oplus D$, we get
\begin{align}\label{addedeqn2}
    \tilde{S}^T(A+H)\tilde{S} = A + \mathcal{O}(\|H\|).
\end{align}
The symplectic matrix $\tilde{S}$ satisfies
\begin{align*}
    \|\tilde{S}\|^2
        &= \|(A+H)^{-1/2}(A+H)^{1/2}\tilde{S} \|^2 \\
        &\leq \|(A+H)^{-1/2}\|^2 \|(A+H)^{1/2}\tilde{S} \|^2 \\
        &= \|(A+H)^{-1}\| \| \tilde{S}^T(A+H)\tilde{S} \|\\
        &=2\|(A+H)^{-1}\| d_{1}(A+H)\\
        &\leq 2\|(A+H)^{-1}\| \|A+H\| = 2\kappa(A+H),
\end{align*}
where $\kappa(T)=\|T\|\|T^{-1}\|$ is the condition number of an invertible matrix $T$, and we used \cite[Lemma 2.2 (iii)]{jm} in the last inequality.
It thus implies that $\|\tilde{S}\|$ is uniformly bounded for small $\|H\|$, which follows from the continuity of $\kappa$.
So, from \eqref{addedeqn2} and the symplectic relation $\tilde{S}^{-T}=J_{2n}\tilde{S}J_{2n}^T$, we get
\begin{equation}\label{addedeqn4}
A\tilde{S} = J_{2n}\tilde{S}J_{2n}^TA+ \mathcal{O}(\|H\|).
\end{equation}
Consider $\tilde{S}$ in the block matrix form:
\begin{align*}
    \tilde{S}= \begin{pmatrix}
            \tilde{W} & \tilde{X} \\
            \tilde{Y} & \tilde{Z}
        \end{pmatrix},
\end{align*}
where each block $\tilde{W}, \tilde{X}, \tilde{Y}, \tilde{Z}$ has size $n \times n$.
From \eqref{addedeqn4} and using the fact $A=D \oplus D$, we get
\begin{align}
    \begin{pmatrix}
            D\tilde{W} & D\tilde{X} \\
            D\tilde{Y} & D\tilde{Z}
        \end{pmatrix} 
        &= \begin{pmatrix}
            0_n & I_n \\
            -I_n & 0_n
        \end{pmatrix} 
        \begin{pmatrix}
            \tilde{W} & \tilde{X} \\
            \tilde{Y} & \tilde{Z}
        \end{pmatrix}
        \begin{pmatrix}
            0_n & -I_n \\
            I_n & 0_n
        \end{pmatrix} 
        \begin{pmatrix}
            D & 0_n \\
            0_n & D
        \end{pmatrix} + \mathcal{O}(\|H\|) \nonumber \\
        &= \begin{pmatrix}
            \tilde{Z}D & -\tilde{Y}D \\
            -\tilde{X}D & \tilde{W}D
        \end{pmatrix} + \mathcal{O}(\|H\|). \label{per_sym_rel}
\end{align}
Now, using the representation $D = \mu_1 I_{|\alpha_1|} \oplus \cdots \oplus \mu_r I_{|\alpha_r|}$, and comparing the corresponding blocks on both sides in \eqref{per_sym_rel}, we get for all $1 \leq i,j \leq r$,
\begin{align}
    \begin{pmatrix}
            \mu_i \tilde{W}_{\alpha_i \alpha_j} & \mu_i \tilde{X}_{\alpha_i \alpha_j} \\
            \mu_i \tilde{Y}_{\alpha_i \alpha_j} & \mu_i \tilde{Z}_{\alpha_i \alpha_j}
        \end{pmatrix} 
        &= \begin{pmatrix}
            \mu_j \tilde{Z}_{\alpha_i \alpha_j} & -\mu_j \tilde{Y}_{\alpha_i \alpha_j} \\
            -\mu_j \tilde{X}_{\alpha_i \alpha_j} & \mu_j \tilde{W}_{\alpha_i \alpha_j}
        \end{pmatrix} + \mathcal{O}(\|H\|). \label{per_sym_rel_2}
\end{align}
This can be equivalently represented as
\begin{align}
\mu_i \tilde{S}_{\gamma_i \gamma_j} 
    &= \mu_j J_{2|\alpha_i|}\tilde{S}_{\gamma_i \gamma_j}J^T_{2|\alpha_j|}+ \mathcal{O}(\|H\|).\label{nq1}
\end{align}
This also gives
\begin{align}
\mu_j \tilde{S}_{\gamma_i \gamma_j} 
    &= \mu_i J_{2|\alpha_i|}\tilde{S}_{\gamma_i \gamma_j} J^T_{2|\alpha_j|}+ \mathcal{O}(\|H\|).\label{nq2}
\end{align}
Adding \eqref{nq1} and \eqref{nq2}, and then dividing by $\mu_i+\mu_j$, gives 
\begin{align}\label{nq3}
\tilde{S}_{\gamma_i \gamma_j} 
    &= J_{2|\alpha_i|}\tilde{S}_{\gamma_j \gamma_j}J^T_{2|\alpha_j|}+ \mathcal{O}(\|H\|).
\end{align}
Suppose we have $i\neq j$. This implies $\mu_i \neq \mu_j.$ By subtracting \eqref{nq2} from \eqref{nq1}, and then dividing by $\mu_i-\mu_j$, we then get 
\begin{align}\label{nq4}
\tilde{S}_{\gamma_i \gamma_j} 
    &=- J_{2|\alpha_i|}\tilde{S}_{\gamma_i \gamma_j}J^T_{2|\alpha_j|}+ \mathcal{O}(\|H\|).
\end{align}
 By adding \eqref{nq3} and \eqref{nq4}, we get $\tilde{S}_{\gamma_i \gamma_j}=\mathcal{O}(\|H\|)$.
 This settles \eqref{nq5}.

We get \eqref{nq6} and \eqref{nq7} directly as a consequence of \eqref{per_sym_rel_2} by taking $i=j$.  

By the symplectic relation $\tilde{S}^TJ_{2n}\tilde{S}=J_{2n},$ we get 
\begin{align}
J_{2|\alpha_i|} 
    &= \tilde{S}^T_{\gamma_i}J_{2n}\tilde{S}_{\gamma_i} \nonumber \\ 
    &= \sum_{k=1}^r \tilde{S}_{\gamma_k\gamma_i}^T J_{2|\alpha_k|} \tilde{S}_{\gamma_k\gamma_i} \nonumber \\
    &= \tilde{S}_{\gamma_i\gamma_i}^T J_{2|\alpha_i|} \tilde{S}_{\gamma_i\gamma_i} + \sum_{k\neq i, k=1}^r \tilde{S}_{\gamma_k\gamma_i}^T J_{2|\alpha_k|} \tilde{S}_{\gamma_k\gamma_i}.\label{new3eqn1}
\end{align}
We know by \eqref{nq5} that $\tilde{S}_{\gamma_k\gamma_i} = \mathcal{O}(\|H\|)$ for all $k \neq i$. Using this in the second term of \eqref{new3eqn1}, we get
\begin{align}
    J_{2|\alpha_i|} = \tilde{S}_{\gamma_i\gamma_i}^T J_{2|\alpha_i|} \tilde{S}_{\gamma_i\gamma_i} +  \mathcal{O}(\|H\|^2). \label{pert_sym_relation}
\end{align}
This implies \eqref{nq8}.
The relation \eqref{pert_sym_relation} also gives
\begin{equation}\label{addedeqn7}
\tilde{S}_{\gamma_i \gamma_i}^T J_{2|\alpha_i|} \tilde{S}_{\gamma_i \gamma_i} J_{2|\alpha_i|}^T
    =I_{2|\alpha_i|} +  \mathcal{O}(\|H\|^2).
\end{equation}
The two relations \eqref{nq6} and \eqref{nq7} can be combined and expressed as 
\begin{equation}\label{addedeqn8}
J_{2|\alpha_i|} \tilde{S}_{\gamma_i \gamma_i} J_{2|\alpha_i|}^T=\tilde{S}_{\gamma_i \gamma_i}+\mathcal{O}(\|H\|).
\end{equation}
Substituting \eqref{addedeqn8} in \eqref{addedeqn7} gives
\begin{align*}
    \tilde{S}_{\gamma_i \gamma_i}^T  \tilde{S}_{\gamma_i \gamma_i} =I_{2|\alpha_i|}+  \mathcal{O}(\|H\|).
\end{align*}
This proves the remaining assertion \eqref{nq9}.
\end{proof}

\begin{remark}
    By taking $H=0_{2n,2n}$ in Proposition~\ref{res1}, we observe that $\left(S^{-1}\tilde{S}\right)_{\gamma_i \gamma_j}=0_{2|\alpha_i|, 2|\alpha_j|}$ for $i\neq j$, and that $Q_{[i]}\coloneqq \left(S^{-1}\tilde{S}\right)_{\gamma_i \gamma_i}$ is orthosymplectic for all $i$. This implies $\tilde{S}=S Q$, where $Q=Q_{[1]} \oplus^{\operatorname{s}} \cdots \oplus^{\operatorname{s}} Q_{[r]}$ is orthosymplectic.
    The following result generalizes this observation for arbitrary $H \to 0_{2n}$.
    % that any symplectic matrix $S$ diagonalizing $A=D \oplus D$ must be orthosymplectic. This was also observed in Corollary~5.3 of \cite{jm}.
\end{remark}

\begin{theorem}\label{res2}
Let $A \in \operatorname{Pd}(2n)$ and $H\in \operatorname{Sm}(2n)$ such that $A + H \in \operatorname{Pd}(2n).$ Let $S\in \operatorname{Sp}(2n; A)$ and $\tilde{S} \in \operatorname{Sp}(2n; A+H)$ be arbitrary. Then there exists an orthosymplectic matrix $Q$ of the form 
\begin{align*}
    Q=Q_{[1]}\oplus^{\operatorname{s}} \cdots \oplus^{\operatorname{s}} Q_{[r]},
\end{align*} 
where $Q_{[i]} \in \operatorname{OrSp}(2|\alpha_i|)$  for all $i=1,\ldots,r$, satisfying 
\begin{align*}
    \tilde{S}=SQ+\mathcal{O}(\|H\|).
\end{align*} 
\end{theorem}
\begin{proof}
There is no loss of generality in assuming that $A$ has the diagonal form $A=D \oplus D$ and $S=I_{2n}$.
With this assumption, Proposition~\ref{res1} gives the following representation of $\tilde{S}$ in terms of a symplectic direct sum:
\begin{align}\label{new4eqn1}
    \tilde{S} =  \oplus^{\operatorname{s}}_i \begin{pmatrix} \tilde{S}_{\alpha_i \alpha_i} & \tilde{S}_{\alpha_i \beta_i}  \\ -\tilde{S}_{\alpha_i \beta_i}  & \tilde{S}_{\alpha_i \alpha_i} \end{pmatrix}+\mathcal{O}(\|H\|).
\end{align}
Our strategy is to apply the Gram-Schmidt orthonormalization process to the columns of $\tilde{S}_{\alpha_i \alpha_i}+\iota\tilde{S}_{\alpha_i \beta_i}$ to obtain a unitary matrix of the form $U_{[i]}+\iota V_{[i]}$, where $U_{[i]}$ and $V_{[i]}$ are real matrices,
and then use the representation \eqref{eq:orthosymplectic_matrix_representation} to obtain orthosymplectic matrix $Q_{[i]}$.

Let $x_1,\ldots, x_{|\alpha_i|}$ and $y_1,\ldots, y_{|\alpha_i|}$ be the columns of $\tilde{S}_{\alpha_i \alpha_i}$ and $\tilde{S}_{\alpha_i \beta_i}$ respectively.
Now, apply the Gram-Schmidt orthonormalization process to the complex vectors $x_1+\iota y_1,\ldots, x_{|\alpha_i|}+\iota y_{|\alpha_i|}.$
Let $z_1=x_1+\iota y_1$. Choose $w_1=z_1/\|z_1\|\equiv u_1+\iota v_1$.
By \eqref{nq8} and \eqref{nq9}, we have 
\begin{align*}
    \|z_1\|^2   &=\|x_1\|^2+\|y_1\|^2 \\
                &= \left\|\begin{pmatrix}
                    x_1 \\ -y_1
                \end{pmatrix}\right\|^2 
                =1+\mathcal{O}(\|H\|).
\end{align*}
 This implies
 \begin{align*}
     w_1 = z_1+\mathcal{O}(\|H\|)=x_1+\iota y_1+\mathcal{O}(\|H\|).
 \end{align*}
%   $y_i = x_i + \mathcal{O}(\|H\|)$ for all $i=1,\ldots,j.$
Let $z_{2}= x_2+\iota y_2 -  \langle w_1, x_2+\iota y_2 \rangle w_1$. Choose $w_{2}=z_{2}/\|z_{2}\|\equiv u_2+\iota v_2$ so that $\{w_1,w_2\}$ is an orthonormal set.
By \eqref{nq8} and \eqref{nq9}, we have 
$\langle x_1+\iota y_1, x_2+\iota y_2\rangle = \mathcal{O}(\|H\|)$. This implies
\begin{align*}
    z_{2}   &=  x_2+\iota x_2 -  \langle w_1, x_2+\iota y_2 \rangle             w_1 \\
            &=  y_2+\iota y_2 -  \langle x_1+\iota y_1, x_2+\iota y_2 \rangle w_1 + \mathcal{O}(\|H\|) \\
            &=  x_2+\iota y_2+\mathcal{O}(\|H\|).
\end{align*}
Again, by \eqref{nq8} and \eqref{nq9}, we have $\|z_2\|=1+\mathcal{O}(\|H\|)$, which implies $w_{2}=x_2+\iota y_2+\mathcal{O}(\|H\|)$.
    
By continuing with the Gram-Schmidt process, we get orthonormal vectors $\{w_1, \ldots, w_{2|\alpha_i|}\}=\{u_1+\iota v_1,\ldots, u_{|\alpha_i|}+\iota v_{|\alpha_i|}\}$ such that for all $j=1,\ldots, |\alpha_i|$,
\begin{align}\label{unitary_bigo}
    u_j+\iota v_j =  x_j+\iota y_j+\mathcal{O}(\|H\|).
\end{align}
Let $U_{[i]} \coloneqq [u_1,\ldots, u_{|\alpha_i|}]$, $V_{[i]} \coloneqq [v_1,\ldots, v_{|\alpha_i|}]$ so that $U_{[i]}+\iota V_{[i]}$ is a unitary matrix.
 By \eqref{eq:orthosymplectic_matrix_representation} it then follows that the following matrix 
\begin{align*}
    Q_{[i]}\coloneqq \begin{pmatrix}
        U_{[i]} & V_{[i]} \\ -V_{[i]} & U_{[i]}
        \end{pmatrix}
\end{align*}
is orthosymplectic. 
The relation \eqref{unitary_bigo} thus gives
\begin{align*}
    Q_{[i]}=\begin{pmatrix}
         \tilde{S}_{\alpha_i \alpha_i} & \tilde{S}_{\alpha_i \beta_i}  \\ -\tilde{S}_{\alpha_i \beta_i}  & \tilde{S}_{\alpha_i \alpha_i} \end{pmatrix}+\mathcal{O}(\|H\|).
\end{align*} 
This combined with \eqref{new4eqn1} gives $\tilde{S} = Q+\mathcal{O}(\|H\|)$ where $Q=Q_{[1]} \oplus^{\operatorname{s}} \cdots \oplus^{\operatorname{s}} Q_{[r]}$, which
 completes the proof.
\end{proof}

The matrix $SQ$ in Theorem~\ref{res2} characterizes the set $\operatorname{Sp}(2n; A)$. We state this in the following proposition, proof of which follows directly from
Corollary~5.3 of \cite{jm}. It is also stated as Theorem~3.5 in \cite{son2021symplectic}.
\begin{proposition}\label{sp2na}
    Let $S \in \operatorname{Sp}(2n; A)$ be fixed. Every symplectic matrix $\hat{S} \in \operatorname{Sp}(2n; A)$ is precisely of the form
    \begin{align*}
         \hat{S}=SQ,
    \end{align*}
    where $Q=Q_{[1]} \oplus^{\operatorname{s}} \cdots \oplus^{\operatorname{s}} Q_{[r]}$ such that $Q_{[i]}\in \operatorname{OrSp}(2|\alpha_i|)$ for all $i=1,\ldots, r$.
\end{proposition}

In \cite{idel}, it is shown that if $A$ has no repeated symplectic eigenvalues, then for any fixed $H \in \operatorname{Sm}(2n)$, one can choose $S \in \operatorname{Sp}(2n; A)$ and $S(\varepsilon) \in \operatorname{Sp}(2n; A+\varepsilon H)$ for small $\varepsilon>0$ such that $\|S(\varepsilon)-S\|=\mathcal{O}(\sqrt{\varepsilon})$. We generalize their result to the more general case of $A$ having repeated symplectic eigenvalues. Moreover, we consider the most general perturbation of $A$ and strengthen the aforementioned result. 
\begin{theorem}\label{res4}
    Let $A \in \operatorname{Pd}(2n)$ and $H\in \operatorname{Sm}(2n)$ such that $A + H \in \operatorname{Pd}(2n).$
    Given any $\tilde{S}\in \operatorname{Sp}(2n; A+H)$, there exists $S \in \operatorname{Sp}(2n; A)$ such that \begin{align}
        \|\tilde{S}-S\|= \mathcal{O}(\|H\|).
    \end{align}
\end{theorem}
\begin{proof}
    Let $M \in \operatorname{Sp}(2n; A)$. By Theorem~\ref{res2}, we have
    \begin{align*}
        \tilde{S} = MQ + \mathcal{O}(\|H\|),
    \end{align*}
    where $Q=Q_{[1]} \oplus^{\operatorname{s}} \cdots \oplus^{\operatorname{s}} Q_{[r]}$ such that $Q_{[i]}\in \operatorname{OrSp}(2|\alpha_i|)$ for all $i=1,\ldots, r$. Set $S \coloneqq MQ$ so that $\|\tilde{S}-S\|=\mathcal{O}(\|H\|)$. We also have $S \in \operatorname{Sp}(2n; A)$ which follows from Proposition~\ref{sp2na}.
\end{proof}

We know from Theorem~\ref{res2} that the distance of the symplectic block $\left(S^{-1}\tilde{S}\right)_{\gamma_i \gamma_i}$ from $\operatorname{OrSp}(2|\alpha_i|)$ is $\mathcal{O}(\|H\|)$ for all $i=1,\ldots, r$. Since $ \operatorname{Sp}(2|\alpha_i|) \supset \operatorname{OrSp}(2|\alpha_i|)$, the distance of $\left(S^{-1}\tilde{S}\right)_{\gamma_i \gamma_i}$ from $\operatorname{Sp}(2|\alpha_i|)$ is expected to be even smaller. The following result shows that this distance is $\mathcal{O}(\|H\|^2)$.

Let $W=[u,v]$ be a $2n \times 2$ matrix such that $\operatorname{Range} (W)$
is {\it non-isotropic}, i.e., $u^TJ_{2n}v \neq 0.$
Let $R=\begin{pmatrix}1 & 0 \\ 0 & u^TJ_{2n}v \end{pmatrix}$ and 
$S=WR^{-1}.$ We then have $S \in \operatorname{Sp}(2n, 2)$. The decomposition $W=SR$ is called the elementary
SR decomposition (ESR).
See \cite{salam2005theoretical} for various versions of ESR and their applications in symplectic analogs of the Gram-Schmidt method.

\begin{proposition}\label{res3}
Let $A \in \operatorname{Pd}(2n)$ and $H\in \operatorname{Sm}(2n)$ such that $A + H \in \operatorname{Pd}(2n).$ Let $S\in \operatorname{Sp}(2n; A)$ and $\tilde{S} \in \operatorname{Sp}(2n; A+H)$. For each $i=1,\ldots,r$,  there exists $N_{[i]} \in \operatorname{Sp}(2|\alpha_i|)$ such that 
\begin{align*}
    \left(S^{-1}\tilde{S}\right)_{\gamma_i \gamma_i} = N_{[i]}+\mathcal{O}(\|H\|^2).
\end{align*}
\end{proposition}
\begin{proof}
Without loss of generality, we can assume that $A$ has the diagonal form $A=D \oplus D$ and $S=I_{2n}$.
Let $u_1,\ldots,u_{|\alpha_i|}, v_1, \ldots, v_{|\alpha_i|}$ be the columns of $\tilde{S}_{\gamma_i \gamma_i}$. Set $M_{[j]} \coloneqq [u_j,v_j]$ for $j=1,\ldots,|\alpha_i|$. 
We will apply mathematical induction on $j$ to construct $N_{[i]}$.
We note that  $\tilde{S}_{\gamma_i \gamma_i}$ can be expressed as
\begin{align*}
    \tilde{S}_{\gamma_i \gamma_i} = M_{[1]} \diamond \cdots \diamond M_{[|\alpha_i|]}.
\end{align*}

Choose $W_{[1]}=M_{[1]}.$ 
We know from \eqref{nq8} that $\operatorname{Range}(W_{[1]})$ is non-isotropic for small $\|H\|$.
Apply ESR to $W_{[1]}$ to get $W_{[1]} = S_{[1]}R_{[1]}$, where
\begin{align}\label{eq:R_[1]}
R_{[1]}&=\begin{pmatrix} 1 & 0 \\ 0 & u_1^TJ_{2|\alpha_i|}v_1\end{pmatrix},
\end{align}
and  $S_{[1]}=W_{[1]} R_{[1]}^{-1} \in \operatorname{Sp}(2|\alpha_i|,2)$.
By \eqref{nq8}, we have $u_1^TJ_{2|\alpha_i|}v_1=1+\mathcal{O}(\|H\|^2)$. 
Substituting this in \eqref{eq:R_[1]} gives
\begin{align}
R_{[1]}&=I_2 + \mathcal{O}(\|H\|^2) \label{addedeqn12}.
\end{align}
Substituting the value of $R_{[1]}$ from \eqref{addedeqn12} in $W_{[1]} = S_{[1]}R_{[1]}$ gives 
 \begin{align*}
     M_{[1]}=W_{[1]}=S_{[1]} + \mathcal{O}(\|H\|^2).
 \end{align*}
 Our induction hypothesis is that, for $1\leq j < |\alpha_i|$, there exist $2|\alpha_i|\times 2$ real matrices $S_{[1]},\ldots, S_{[j]}$  satisfying $S_{[1]} \diamond \cdots \diamond S_{[j]} \in \operatorname{Sp}(2|\alpha_i|, 2j)$
and 
\begin{align}\label{nq14}
M_{[1]} \diamond \cdots \diamond M_{[j]} &= S_{[1]} \diamond \cdots \diamond S_{[j]} + \mathcal{O}(\|H\|^2).
\end{align}
We choose
\begin{align}\label{addedeqn18}
W_{[j+1]}=M_{[j+1]}-\left(S_{[1]} \diamond \cdots \diamond S_{[j]} \right)J^T_{2j}
\left(S_{[1]} \diamond \cdots \diamond S_{[j]} \right)^TJ_{2|\alpha_i|}M_{[j+1]}.
\end{align}
By \eqref{nq8} and \eqref{nq14}  we have
\begin{equation}\label{nq15}
W_{[j+1]}=M_{[j+1]}+\mathcal{O}(\|H\|^2),
\end{equation}
which implies
$\operatorname{Range}(W_{[j+1]})$ is non-isotropic for small $\mathcal{O}(\|H\|)$. Apply ESR to $W_{[j+1]}=[w_{j+1}, z_{j+1}]$ to 
get $W_{[j+1]} = S_{[j+1]}R_{[j+1]}$. Here $S_{[j+1]} \in \operatorname{Sp}(2|\alpha_i|,2)$ and 
\begin{align}
    R_{[j+1]} &=\begin{pmatrix} 1 & 0 \\ 0 & w_{j+1}^TJ_{2|\alpha_i|}z_{j+1}\end{pmatrix}.\label{addedeqn17}
\end{align}
From \eqref{nq8} and \eqref{nq15}, we get $w_{j+1}^TJ_{2|\alpha_i|}z_{j+1}=1+ \mathcal{O}(\|H\|^2)$.
Using this relation in \eqref{addedeqn17}  implies $R_{[j+1]}=I_2 + \mathcal{O}(\|H\|^2)$. Substituting this in $W_{[j+1]} = S_{[j+1]}R_{[j+1]}$ gives 
\begin{align}\label{addedeqn21}
    W_{[j+1]} &= S_{[j+1]}+\mathcal{O}(\|H\|^2).
\end{align}
Combining \eqref{nq15} and \eqref{addedeqn21} then gives
\begin{align*}
M_{[j+1]}=S_{[j+1]}+ \mathcal{O}(\|H\|^2).
\end{align*}
We thus have
\begin{align*}
    M_{[1]} \diamond \cdots \diamond M_{[j+1]} &= S_{[1]} \diamond \cdots \diamond S_{[j+1]} + \mathcal{O}(\|H\|^2).
\end{align*}
To complete the induction, we just need to show that 
$S_{[1]} \diamond \cdots \diamond S_{[j+1]} \in \operatorname{Sp}(2|\alpha_i|, 2(j+1))$.
We have
\begin{align*}
    S_{[1]} \diamond \cdots \diamond S_{[j+1]} &= \left(S_{[1]} \diamond \cdots \diamond S_{[j]}\right) \diamond S_{[j+1]}.
\end{align*}
By the necessary and sufficient condition for $(S_{[1]} \diamond \cdots \diamond S_{[j]}) \diamond S_{[j+1]} \in \operatorname{Sp}(2|\alpha_i|, 2(j+1))$, as discussed in Section~\ref{sym_concatenation},
 it is equivalent to show that $(S_{[1]} \diamond \cdots \diamond S_{[j]} )^T J_{2|\alpha_i|}S_{[j+1]}$ is the zero matrix.
Now, using the relation $W_{[j+1]} = S_{[j+1]}R_{[j+1]}$ we get
\begin{align}\label{final_induction_1}
    \left(S_{[1]} \diamond \cdots \diamond S_{[j]} \right)^T J_{2|\alpha_i|}S_{[j+1]}  &= \left(S_{[1]} \diamond \cdots \diamond S_{[j]} \right)^T J_{2|\alpha_i|}W_{[j+1]}R_{[j+1]}^{-1}.
\end{align}
Substitute in \eqref{final_induction_1}  the value of $W_{[j+1]}$ from \eqref{addedeqn18} to get
\begin{multline*}
    \left(S_{[1]} \diamond \cdots \diamond S_{[j]} \right)^T J_{2|\alpha_i|}S_{[j+1]}
    =\left(S_{[1]} \diamond \cdots \diamond S_{[j]} \right)^T J_{2|\alpha_i|} \\ \left[M_{[j+1]}-\left(S_{[1]} \diamond \cdots \diamond S_{[j]} \right)J^T_{2j}
\left(S_{[1]} \diamond \cdots \diamond S_{[j]} \right)^TJ_{2|\alpha_i|}M_{[j+1]}\right]R_{[j+1]}^{-1}.
\end{multline*}
Apply the induction hypothesis $S_{[1]} \diamond \cdots \diamond S_{[j]} \in \operatorname{Sp}(2|\alpha_i|, 2j)$ and simplify as follows:
\begin{align}
    &\left(S_{[1]} \diamond \cdots \diamond S_{[j]} \right)^T J_{2|\alpha_i|}S_{[j+1]} \nonumber \\    
    &=\left[\left(S_{[1]} \diamond \cdots \diamond S_{[j]} \right)^T J_{2|\alpha_i|} M_{[j+1]}-J_{2j} J^T_{2j}
\left(S_{[1]} \diamond \cdots \diamond S_{[j]} \right)^TJ_{2|\alpha_i|}M_{[j+1]} \right]R_{[j+1]}^{-1} \nonumber  \\
    &=\left[\left(S_{[1]} \diamond \cdots \diamond S_{[j]} \right)^T J_{2|\alpha_i|} M_{[j+1]}-
\left(S_{[1]} \diamond \cdots \diamond S_{[j]} \right)^TJ_{2|\alpha_i|}M_{[j+1]} \right]R_{[j+1]}^{-1} \nonumber \\
&= 0_{2j, 2} \label{addedeqn20}.
\end{align}
We have thus shown that $S_{[1]} \diamond \cdots \diamond S_{[j+1]} \in \operatorname{Sp}(2|\alpha_i|, 2(j+1))$.
By induction, we then get the desired matrix $N_{[i]}=S_{[1]} \diamond \cdots \diamond S_{[|\alpha_i|]}  \in \operatorname{Sp}(2|\alpha_i|)$, which satisfies
\begin{align*}
    \tilde{S}_{\gamma_i \gamma_i} = M_{[1]} \diamond \cdots \diamond M_{[|\alpha_i|]} = N_{[i]}+\mathcal{O}(\|H\|^2).
\end{align*}
\end{proof}

%%%%%%%%%%%%%%%%%%%%%%%%%%%%%%%%%%%%%%%%%%%%%%%%%%%%%%%%%%%%%%%%%%

\section{Conclusion}
 One of the main findings of our work is that, given any $S\in \operatorname{Sp}(2n; A)$ and $\tilde{S} \in \operatorname{Sp}(2n; A+H)$, there exists an orthosymplectic matrix $Q$ such that $\tilde{S}=SQ+\mathcal{O}(\|H\|)$. Moreover, the orthosymplectic matrix $Q$ has structure $Q=Q_{[1]} \oplus^s \cdots \oplus^s Q_{[r]}$, where $Q_{[j]}$ is a $2|\alpha_j| \times 2 |\alpha_j|$ orthosymplectic matrix. Here $r$ is the number of distinct symplectic eigenvalues $\mu_1,\ldots, \mu_r$ of $A$ and $\alpha_j$ is the set of indices of the symplectic eigenvalues of $A$ equal to $\mu_j$.  We also proved that $S\in \operatorname{Sp}(2n; A)$ and $\tilde{S} \in \operatorname{Sp}(2n; A+H)$  can be chosen so that $\|\tilde{S}-S\|=\mathcal{O}(\|H\|)$.

%%%%%%%%%%%%%%%%%%%%%%%%%%%%%%%%%%%%%%%%%%%%%%%%%%%%%%%%%%%%%%%%%%

\section*{Acknowledgments}
The authors are grateful to Prof.~Tanvi Jain for the insightful discussions that took place in the initial stage of the work. HKM acknowledges the National Science Foundation under Grant No.~2304816 for financial support. The authors thank Prof.~Mark M.~Wilde for pointing out some mistakes during the preparation of the manuscript and Dr.~Tiju Cherian John for some critical comments. The authors are thankful to the anonymous referee for their thoughtful comments and suggestions that improved the readability of the paper.

 %%%%%%%%%%%%%%%%%%%%%%%%%%%%%%%%%%%%%%%%%%%%%%%%%%%%%%%%%%%%%%%%%%

%%%%%%%%%%%%%%%%%%%%%%%%%%%%%%%%%%%%%%%%%%%%%%%%%%%%%%%%%%%%%%%%%%
%%%%%%%%%%%%%%%%%%%%%%%%%%%%%%%%%%%%%%%%%%%%%%%%%%%%%%%%%%%%%%%%%%
%\clearpage
 \bibliographystyle{unsrt}
\bibliography{block_symplectic_perturbation}

%%%%%%%%%%%%%%%%%%%%%%%%%%%%%%%%%%%%%%%%%%%%%%%%%%%%%%%%%%%%%%%%%%
%%%%%%%%%%%%%%%%%%%%%%%%%%%%%%%%%%%%%%%%%%%%%%%%%%%%%%%%%%%%%%%%%%
%%%%%%%%%%%%%%%%%%%%%%%%%%%%%%%%%%%%%%%%%%%%%%%%%%%%%%%%%%%%%%%%%%
%%%%%%%%%%%%%%%%%%%%%%%%%%%%%%%%%%%%%%%%%%%%%%%%%%%%%%%%%%%%%%%%%%
%%%%%%%%%%%%%%%%%%%%%%%%%%%%%%%%%%%%%%%%%%%%%%%%%%%%%%%%%%%%%%%%%%
%%%%%%%%%%%%%%%%%%%%%%%%%%%%%%%%%%%%%%%%%%%%%%%%%%%%%%%%%%%%%%%%%%
%%%%%%%%%%%%%%%%%%%%%%%%%%%%%%%%%%%%%%%%%%%%%%%%%%%%%%%%%%%%%%%%%%

%%%%%%%%%%%%%%%%%%%%%%%%%%%%%%%%%%%%%%%%%%%%%%%%%%%%%%%%%%%%%

\end{document}